\documentclass[leqno,12pt]{amsart}
\usepackage{amsmath,amssymb,amsthm}

\theoremstyle{plain} 
\newtheorem{theorem}{\indent\sc Theorem}[section]
\newtheorem{lemma}[theorem]{\indent\sc Lemma}

\newtheorem{proposition}[theorem]{\indent\sc Proposition}

\theoremstyle{definition} 

\newtheorem{remark}[theorem]{\indent\sc Remark}

\numberwithin{equation}{section}

\usepackage{bbm,calligra}
\usepackage{mathrsfs}
\usepackage[all]{xy}

\newcommand{\g}{\mathfrak{g}}
\newcommand{\fa}{\mathfrak{a}}

\newcommand{\h}{\mathfrak{h}}
\newcommand{\uG}{\mathbf{G}}
\newcommand{\bG}{\mathbf{G}}
\newcommand{\fG}{\mathfrak{G}}

\newcommand{\Lie}{\mathrm{Lie}}

\newcommand{\fLie}{\mathop{\mathfrak{Lie}}}
\newcommand{\Z}{\mathbb{Z}}
\newcommand{\K}{\mathbb{K}}
\newcommand{\C}{\mathbb{C}}
\newcommand{\Q}{\mathbb{Q}}
\newcommand{\F}{\mathbb{F}}

\newcommand{\Aut}{\mathrm{Aut}}
\newcommand{\bAut}{\mathbf{Aut}}
\newcommand{\Out}{\mathrm{Out}}
\newcommand{\bOut}{\mathbf{Out}}
\newcommand{\Isom}{\mathbf{Isom}}
\newcommand{\Hom}{\mathrm{Hom}}

\newcommand{\hA}{\widehat{A}}

\newcommand{\et}{\text{\rm \'et}}

\begin{document}
\title[Affine Kac-Moody groups and Lie algebras in the language of SGA3]{Affine Kac-Moody groups and Lie algebras in the language of SGA3}

\author[J.~Morita]{Jun Morita}
\author[A.~Pianzola]{Arturo Pianzola}
\author[T.~Shibata]{Taiki Shibata}
\dedicatory{Dedicated to Professor~Robert~V.~Moody on the~occasion of his~80th~birthday.}

\thanks{
J.M. is supported by JSPS KAKENHI (Grants-in-Aid for Scientific Research) Grant Number~JP17K05158.
A.P. acknowledges the continuous support of NSERC and Conicet, and the wonderful hospitality of Tsukuba University.
T.S. was a Pacific Institute for the Mathematical Sciences (PIMS) Postdoctoral Fellow at the University of Alberta, and is supported by JSPS KAKENHI Grant Numbers~JP19K14517 and JP22K13905.
}

\address[J.~Morita]{%
Institute of Mathematics\endgraf
University of Tsukuba\endgraf
1-1-1 Tennodai, Tsukuba, Ibaraki 305-8571\endgraf
Japan}
\email{morita@math.tsukuba.ac.jp}

\address[A.~Pianzola]{%
Department of Mathematical and Statistical Sciences\endgraf
University of Alberta\endgraf
Edmonton, Alberta T6G 2G1\endgraf
Canada\endgraf
and\endgraf
Centro de Altos Estudios en Ciencias Exactas\endgraf
Avenida de Mayo 866, (1084), Buenos Aires\endgraf
Argentina}
\email{a.pianzola@ualberta.ca}

\address[T.~Shibata]{%
Department of Applied Mathematics\endgraf
Okayama University of Science\endgraf
1-1 Ridai-cho Kita-ku, Okayama, Okayama 700-0005\endgraf
Japan}
\email{shibata@ous.ac.jp}

\maketitle

\begin{abstract}
In infinite-dimensional Lie theory, the affine Kac-Moody Lie algebras and groups play a distinguished role due to their many applications to various areas of mathematics and physics. Underlying these infinite-dimensional objects there are closely related group schemes and Lie algebras of finite type over Laurent polynomial rings. The language of SGA3 is perfectly suited to describe such objects. The purpose of this short article is to  provide a natural description of the affine Kac-Moody groups and Lie algebras using this language.
\end{abstract}

\section{Recollections of affine Kac-Moody groups and Lie algebras}
Kac-Moody Lie algebras were independently discovered by V.~Kac \cite{Kac1} and R.~V.~Moody \cite{Mdy1}. Given a field $\K$ of characteristic $0$, they are defined by generators and relations (\`a la Chevalley-Harish-Chandra-Serre) encoded in a generalized Cartan matrix (GCM) $A$. If $A$ is a Cartan matrix of type $X$, then the corresponding Kac-Moody Lie algebra is nothing but the split simple finite-dimensional $\K$-Lie algebra of type $X$. Closely related to these are the affine Lie algebras, see \cite[Chapter7]{Kac2} and \cite{Mdy2}.

The definition of ``simply connected'' Kac-Moody groups over fields of characteristic $0$ appears in \cite{PK}. This paper establishes the conjugacy theorem of ``Cartan subalgebras'' of symmetrizable Kac-Moody Lie algebras and, as a consequence, that the GCMs and corresponding root systems are an invariant of the algebras. Detailed expositions of this material are given in \cite{Kmr} and \cite{MP}.

If $A$ is of finite type $X$, the corresponding ``groups'' (with the simply connected being the largest) exist (Chevalley) and are unique (Demazure). They are smooth group schemes over $\Z$ constructed using a ``root datum'' that includes $A$. The base change to $\K$ produces a linear algebraic group over $\K$ whose Lie algebra is split simple of type $X$ if $\K$ is of characteristic $0$. The first incursion into Kac-Moody groups  for $\K$ arbitrary is given in \cite{MT}. A clear exposition of the theory at that time is given in \cite{Tit1}.

J.~Tits pioneered the idea of defining root datum based on GCMs, and attaching to them group functors that ``behave right'' when evaluated at arbitrary fields \cite{Tit2}. See also \cite{Rem}. Some of the affine cases are discussed in examples. Further clarity about the nature of the abstract groups obtained in this fashion is given by the construction of Kac-Moody groups by means of ind-schemes and flag varieties due to O.~Mathieu in \cite{Mth}. Work related to this fertile approach include \cite{HLR}, \cite{L} and \cite{PR}.\footnote{\, We are focusing on algebraic/geometric treatments. In addition to these, there are many topological and analytical approaches to the topic. These will not be discussed.}

While the works described above deal with all infinite-dimensional Kac-Moody groups and algebras, our focus is the affine case using the techniques of \cite{SGA3}. Our ``groups'' are group schemes. Any such object has a Lie algebra defined by means of dual numbers and --\,as it should be, and in analogy with the classical theory\,-- this is the way that the affine algebras appear when the base field is of characteristic $0$.
%

\subsection{} \label{section1.1}
Throughout this paper $X_N$ will denote the type and rank of an irreducible indecomposable finite root system in the sense of Bourbaki \cite{Bo}. Thus
\[
X_N = A_{N \geq 1},\, B_{N\geq 2},\, C_{N \geq3},\, D_{N \geq 4},\, G_2,\, F_4,\, E_6,\, E_7,\, \text{or} \,\, E_8.
\]

To each of these types corresponds a unique up to isomorphism simply-connected Chevalley-Demazure group scheme $\bG$ over $\Z$, and a split simple finite-dimensional Lie algebra $\g$ over $\Q$. With the notation of \cite{DG}, we have
\[
\Lie(\bG) \otimes_\Z \Q \simeq \g.
\]

By \cite[Exp.~XXIV~Theo.~1.3]{SGA3} we have a split exact sequence of affine $\Z$-groups
\begin{equation}\label{exact}
1 \longrightarrow \bG_{\rm ad} \longrightarrow \bAut(\uG) \overset{\rho}{\longrightarrow} \bOut(\g) \longrightarrow 1,
\end{equation}
where $\bG_{\rm ad}$ is the adjoint group of $\bG$ and $\bOut(\g)$ is the finite constant group over $\Z$ corresponding to the finite (abstract) group $\Out(\g)$ of symmetries of the Coxeter-Dynkin diagram of $\g$.

We once and for all fix a section of $\rho:\bAut(\uG)\to\bOut(\g)$ and a Killing couple $(\mathbf{B},\mathbf{H})$ of $\bG$ (see \cite[Exp. XXII 5.3.13]{SGA3}) so that $\bOut(\g)$ is viewed as ``diagram automorphisms'' of $\uG$ (see \cite[Exp. XXII Cor. 5.5.5]{SGA3} and \cite[Exp. XXIV Theo. 1.3(iii)]{SGA3}). That is, if $\Delta$ is the root system and $\Pi$ the base of $\Delta$ corresponding to $(\mathbf{B},\mathbf{H})$, and if $\sigma \in \bOut(\g)(\Z) = \bOut(\g)(\Q) = \Out(\g)$, then $\sigma$ (viewed as an automorphism of $\uG$ via our section) stabilizes $\mathbf{H}$ and its transpose permutes the elements of $\Pi$ in the way that $\sigma$ does.\footnote{\, In \cite{SGA3} the group $\bOut(\g)$ is denoted by ${\rm \bf Autext}(\bG).$ The content of the assertions of this last sentence is \cite[Exp. XXIV Theo. 1.3(iii)]{SGA3} which was mentioned above.}

 If we fix $(\mathbf{B}, \mathbf{H})$
and our section $\rho$ all of the above considerations hold for $\bG_R$ for any ring $R$ and are functorial on $R.$

The abelian Lie algebra $\h = \Lie(\mathbf{H}) \otimes_\Z \Q$ is a split Cartan subalgebra of $\g$. Similar considerations apply to $\sigma$ viewed as an automorphism of $\g$: It stabilizes $\h$ and permutes the elements of $\Pi$ as above.

\begin{remark}
This dual interpretation of the elements of $\Out(\g)$ is present throughout our work. In every case it will be clear which of the two interpretations is used. 
\end{remark}

\begin{remark}\label{Z}
By base change the exact sequence \eqref{exact} exists if we replace $\Z$ by an arbitrary ring $R$.
\end{remark}

\subsection{}
Following Kac's notation \cite{Kac2}, the sixteen types of affine Kac-Moody Lie algebras and groups will be denoted by $X_N^{(r)}$. The possibilities are
\[
A_1^{(1)},\, A_{N \geq 2}^{(1,2)},\, B_{N \geq 2}^{(1)},\, C_{N \geq 2}^{(1)},\, D_{4}^{(1,2,3)},\, D_{N \geq 5}^{(1,2)},\,\, G_2^{(1)},\,\, F_4^{(1)},\,\, E_6^{(1,2)},\,\, E_7^{(1)},\, \text{or} \,\, E_8^{(1)}.
\]

The $r$ appearing in the $X_N^{(r)}$ is the order of an element $\sigma\in\Out(\g)$. Such an element is unique up to conjugacy and can be used to realize the affine algebras (derived modulo their centres) as loop algebras over the complex numbers. These are also Lie algebras over the ring $\C[t^{\pm1}]$. As such, they are of absolute type $X_N$; i.e., they become split of type $X_N$ by base change to an \'etale covering $\C[t^{\pm\frac{1}{r}}]$ of $\C[t^{\pm1}]$. See \cite{P1} for details. From this point of view the sixteen affine types are denoted by the pairs $(\g,\sigma)$, where $\g$ is a finite dimensional split simple Lie algebra, and $\sigma$  a Dynkin diagram automorphism of $\g$ of order $r$. We will use both notations since each has its own advantages depending on the context.

\begin{remark}
An alternative notation for the affine Lie algebras is that of \cite{MP} which is closer to the work of Tits on the classification of semisimple algebraic groups \cite{Tit4}. For example the loop algebra of type $D_4^{(3)}$ in Kac's notation, is of type $G_2^{(3)}$ in \cite{MP}. The reason is that when passing to the generic fibre of the corresponding loop algebra, we obtain a simple Lie algebra over $\C(t)$ where the maximal abelian diagonalizable subalgebras are 2-dimensional and the relative root system is of type $G_2$. By passing to $\C(t^{\frac{1}{3}})$ the Lie algebra becomes split of type $D_4$. 
\end{remark}
\medskip

\centerline {***}
\medskip

In what follows $\F$ will denote a prime field. Thus either $\F=\Q$ or $\F=\F_p$ for some prime number $p$. Every field $\K$ contains a (unique up to unique isomorphism) prime subfield.

If $\K$ is of characteristic $0$ we let $\hat \g_\K(X_N^{(r)})$ denote the affine Kac-Moody Lie algebra of type $X_N^{(r)}$ over $\K$. The derived algebra of $\hat\g_\K(X_N^{(r)})$ has a 1-dimensional centre, and the corresponding central quotient will be denoted by $\g_\K(X_N^{(r)})$.
\smallskip

The corresponding affine Kac-Moody group of simply connected type can be defined for $\K$ arbitrary (see \cite{Tit3} and \cite{MorPiaShi21}) and will be denoted by $\hat G_\K(X_N^{(r)})$.\footnote{\, More precisely, $\hat G_\mathbb{K}(X_N^{(r)})$ is the subgroup of $\mathrm{Aut}_\mathbb{K}(V_\mathbb{K})$ generated by $x_\beta(\nu)$ for all $\beta \in \Delta^{\rm re}$ and for all $\nu \in \mathbb{K}$. Here $V_\mathbb{K} = \mathbb{K} \otimes_{\mathbb{Z}} V_\mathbb{Z}$ is defined by a certain admissible pair $(V,V_\mathbb{Z})$ of simply connected type, and $\Delta^{\rm re}$ is the set of real roots of the corresponding Kac-Moody Lie algebra (cf.~\cite[A.1-A.5]{MorPiaShi21}).}

\section{Statement of the Main Theorem}
The main result to be established in this paper is the following.
\begin{theorem} \label{MainShort}
Let $X_N^{(r)} = (\g, \sigma)$ be an affine type. Let $\F$ be a prime field of characteristic different than $r,$ and $\K$ an arbitrary field extension of $\F$. There exists a simply connected semisimple group scheme $\uG_{(\g,\sigma)}$ over $\F[t^{\pm1}]$ such that:
\begin{enumerate}
\item
$\uG_{(\g,\sigma)}(\K[t^{\pm1}]) \simeq \hat G_\K(X_N^{(r)})/\K^\times$.
\item
If $\F=\Q$, there exist a natural $\K[t^{\pm1}]$-Lie algebra isomorphism 
\[
\fLie (\uG_{(\g,\sigma)})(\K[t^{\pm1}]) \simeq \g_{\K}(X_N^{(r)}).
\]
Furthermore, $\uG_{(\g,\sigma)}$ is unique up to isomorphism.
\end{enumerate}
\end{theorem}
The terminology and notation for group schemes and their Lie algebras are that of \cite{DG} and \cite{SGA3}. That the infinite-dimensional $\K$-Lie algebra $\g_{\K}(X_N^{(r)})$ has a natural $\K[t^{\pm1}]$ structure will be explained in \S\ref{Section:Proof}.

\section{The affine groups in terms of fixed points} \label{Section:Affine}
Let $\uG$ and $\g$ be as in Section~\ref{section1.1},  and let $\F$ and $\K$ be as above. Set $R_\K =\K[t^{\pm1}]$. 

The abstract groups $\uG(R_\K)$ of the various $\uG$ are the so-called $\K$-{\it loop groups}.  In this section, we introduce the twisted version of $\K$-loop groups of $\uG$. To construct them, we use our Dynkin diagram automorphism $\sigma$ of order $r$ and the Galois extensions $S_{(\K)}$  of $R_\K$ with Galois group $\Gamma$ defined below.
\medskip

\noindent{\bf Case~I:}
$\sigma$ is of order $r=1$ (i.e., $\sigma=\mathrm{id}$). We let $S_{(\K)}=R_\K = \K[t^{\pm1}]$ and $\Gamma=1$.
\medskip

\noindent{\bf Case~II:}
$\sigma$ is of order $r=2$ and $\F \neq \F_2$. We let $S_{(\K)}:=\K[t^{\pm {\frac{1}{2}}}]$, and let $\sigma'$ be the unique $\K$-algebra automorphism $S_{(\K)}$ defined by $\sigma'(t^{\frac{1}{2}}) = -t^{\frac{1}{2}}$. Set $\Gamma:=\langle\sigma'\rangle \simeq \Z/2\Z$.
\medskip

\noindent{\bf Case~III:}
$\sigma$ is of order $r=3$ (in particular $\g$ is of type $D_4$) and $\F \neq  \F_3$. We fix an element (unique up to conjugation) $\omega\in\Out(\g)$ of order $2$. Then $\sigma$ and $\omega$ generates $\Out(\g) \simeq S_3$.

Let $\xi$ be a primitive third root of unity in a fixed algebraic closure $\overline{\F}$ of $\F$. Without loss of generality we fix an algebraic closure $\overline{\K}$ of $\K$ containing $\overline{\F}$. The discussion is divided into two cases:

{\bf (a)} $\xi\in\K$.
In this case, let $S_{(\K)}:=\K[t^{\pm \frac{1}{3}}]$, and let $\sigma'$ be the unique $\K$-algebra automorphism of $S_{(\K)}$ defined by $\sigma'(t^{\frac{1}{3}})=\xi t^{\frac{1}{3}}$. We set $\Gamma:= \langle\sigma'\rangle \simeq \Z/3\Z$.

{\bf (b)} $\xi\not\in\K$.
In this case, we set $S_{(\K)}:=\K(\xi)[t^{\pm \frac{1}{3}}]$, where $\K(\xi)$ is the quadratic field extension obtained by adjoining $\xi$ to $\K$. Let $\sigma'$ be the unique $\K(\xi)$-algebra automorphism of $S_{(\K)}$ such that $\sigma'(t^{\frac{1}{3}}) = \xi t^{\frac{1}{3}}$, and let $\omega'$ be the unique $\K[t^{\pm \frac{1}{3}}]$-algebra automorphism $S_{(\K)}$ such that $\omega'(\xi)=\xi^2=\xi^{-1}$. We set $\Gamma:=\langle\sigma',\omega'\rangle \simeq S_3$.

We leave it to the reader to check that, thus defined, $S_{(\K)}$ is a Galois extension of $R_\K$ with Galois group $\Gamma$. See \cite{CHR}, \cite{KO} and \cite{Wat} for generalities about Galois extensions of rings.
\smallskip

We are now ready to describe the connection between twisted loop groups and affine Kac-Moody groups.
\smallskip

For a root $\alpha \in \Delta$ and $u\in S_{(\K)}$, we denote the associated unipotent elements of $\uG\big(S_{(\K)}\big)$ by $x_\alpha(u)$. Since $S_{(\K)}$ is a Euclidean ring, it is known that $\uG\big(S_{(\K)}\big)$ is generated by such $x_\alpha(u)$'s, see \cite[Chapter~8]{Ste}. We define a $\Gamma$-action on $\uG\big(S_{(\K)}\big)$ as follows.
\[
\sigma'(x_\alpha(u)) = x_{\sigma(\alpha)}(k_\alpha \sigma'(u)) \quad\text{and}\quad \omega'(x_\alpha(u)) = x_{\omega(\alpha)}(\omega'(u)).
\]
Here, the sign $k_\alpha=\pm1$ is defined in \cite[\S3]{Abe77} when $r=1,2$ and in \cite[\S3]{MorPiaShi21} when $r=3$. The $\Gamma$-fixed points subgroup $\uG\big(S_{(\K)}\big)^\Gamma$ of the various $\uG\big(S_{(\K)}\big)$ are the {\it twisted $\K$-loop groups}.

\begin{remark}\label{iso}
In type $D_4$, the elements $\sigma$ and $\omega$ of $\Out(\g) \simeq S_3$ are unique up to conjugacy. The resulting twisted loop groups are independent, up to isomorphism, of the choice of these two elements.
\end{remark}

\begin{remark}
While the structure of $\Gamma$ is independent of $\K$ for $r=1,2$, this is not the case for $r=3$. To avoid any possible confusion in what follows we will write $\Gamma(\K)$ instead of $\Gamma$ if necessary.
\end{remark}

In \cite[Prop.~6.6]{MorPiaShi21}, we constructed a group homomorphism $\Phi$ from the Kac-Moody group $\hat G_\K(X_N^{(r)})$ to $\uG\big(S_{(\K)}\big)$ whose image coincides with the twisted loop group $\uG\big(S_{(\K)}\big)^{\Gamma(\K)}$. Moreover, it is shown that $\hat G_\K(X_N^{(r)})$ is a 1-dimensional central extension of the twisted $\K$-loop group ({\it ibid.}~Theo.~6.7). This yields
\begin{theorem}\label{MainLong}
There exists a group isomorphism $\hat G_\K(X_N^{(r)})/\K^\times \simeq \uG\big(S_{(\K)}\big)^{\Gamma(\K)}$.
\end{theorem}

\section{Proof of the Main Theorem}\label{Section:Proof}
Let $\uG,$ $\g$, $\F$ and $\K$ be as in the previous Section. Set $R=\F[t^{\pm1}],$ and let $A$ be the $R$-Hopf algebra representing $\uG_R$. 

 \begin{remark}\label{identify} In what follows, we will (anti-)identify without further reference $\bAut(\bG_R)$ with the group of automorphisms $\Aut(A)$ of $R$-Hopf algebra automorphisms of $A$. In particular, the elements of ${\rm Out}(\g) = \bOut(\g)(R)$ are viewed as elements of $\Aut(A)$ (see \S 1.1 above).
 \end{remark}

\subsection{}
We next turn to the construction of $\uG_{(\g,\sigma)}$ linked to $\hat G(X_N^{(r)})$.  We do so by defining its corresponding $R$-Hopf algebra $\hA=R[\bG_{(\g,\sigma)}]$. This we will do by the yoga of Galois descent:
\medskip

{\it Let $S/R$ be a Galois extension with Galois group $\Gamma$. To give a Hopf algebra $\hA$ over $R$ with the property that $\hA\otimes_R S \simeq A\otimes_R S$ {\rm is the same} as to give an action of $\Gamma$ on $A\otimes_R S$ by semilinear Hopf algebra automorphisms.}
\medskip

Consider the Galois extensions of $R$ given in \S\ref{Section:Affine} when $\K = \F$. Thus $R=\F[t^{\pm1}],$ and
\smallskip

$S = R$ and $\Gamma = {1}$ if $r=1$.
\smallskip

$S = \F[t^{\pm {\frac{1}{2}}}]$ and $\Gamma=\langle\sigma'\rangle \simeq \Z/2\Z$ if $r=2$.
\smallskip

$S = \F[t^{\pm \frac{1}{3}}]$ and $\Gamma=\langle \sigma'\rangle \simeq \Z/3\Z$ if $r=3$ and $\xi \in \F$.
\smallskip

$S = \F(\xi)[t^{\pm \frac{1}{3}}]$ and $\Gamma=\langle\sigma',\omega'\rangle \simeq S_3$ if $r=3$ and $\xi \notin \F$.
\medskip
 
 Define an action of $\Gamma$ as a Hopf algebra automorphism of $A\otimes_R S$ via
\[
{}^\gamma(a\otimes s) \,\,=\,\, \gamma(a)\otimes {}^{\gamma'}s \]
for all $a \in A$, $s \in S$ and $\gamma \in \Gamma$.\footnote{\, Recall that by construction $\Gamma$ is identified with an explicit subgroup of $\Out(\g)$. The meaning of $\gamma(a)$ is given by Remark \ref{identify}}
This action is visibly $\Gamma$-semilinear. We define $\hA=(A\otimes_R S)^\Gamma$ and denote by $\uG_{(\g,\sigma)}$ the corresponding $R$-group scheme.
\begin{remark}\label{Reductive} $\uG_{(\g,\sigma)}$ is a reductive $R$-group. Indeed $\hA \otimes_R S \simeq A \otimes_R S.$ Thus $\uG_{(\g,\sigma)} \times_R S \simeq \bG_S.$ Since $\bG_S$ is reductive, so is $\uG_{(\g,\sigma)}$ by faithfully flat descent. In particular, $\uG_{(\g,\sigma)}$ is smooth. 
\end{remark}
\medskip

\begin{lemma} \label{prp:Twi-Sch_lem}
The restriction map
\[
\Hom_S(A\otimes_R S, S) \longrightarrow \Hom_R(\hA,S); \quad f\longmapsto f|_{\hA}
\]
induces an abstract group isomorphism $\uG(S)^\Gamma \overset{\simeq}{\longrightarrow} \uG_{(\g,\sigma)}(R)$.
\end{lemma}
\begin{proof}
Let $f\in \uG(S)^\Gamma$. Then for all $x \in A\otimes_R S$ and $\gamma\in\Gamma$ we have 
\[
f(x) \,\,=\,\,({}^\gamma f)(x) \,\,:=\,\, {}^\gamma(f\circ \gamma^{-1}(x)).
\]
If $x\in \hA$ then $\gamma^{-1}(x)=x$ so that we get
\[
{}^\gamma(f(x)) \,\,=\,\, f(x) \qquad \forall\gamma\in\Gamma.
\]
Thus $f(x)\in R$. In other words,
\[
f|_{\hA} \in \Hom_R(\hA,R) = \uG_{(\g,\sigma)}(R).
\]

As for  the converse, let us first recall that the explicit isomorphism $\mu : \hA\otimes_R S \simeq A\otimes_R S$ given by Galois descent satisfies
\begin{equation}\label{GI}
 x \otimes s \mapsto sx \,\,\ \forall x \in \hA, \,\, s \in S
 \end{equation} 
where $sx$ is the scalar linear action of $s$ on the $S$-module $A \otimes_R S.$

Given  $g\in \Hom_R(\hA,R)$, consider $f=(g\otimes {\rm Id}_S)\circ \mu^{-1}$. Both $f$ and ${}^\gamma f$ are determined by their restriction to $\mu(\hA) := \mu(\hat{A} \otimes 1)\subset A\otimes_R S$ since both maps are  $S$-linear and $\mu(\hA)$ spans $A\otimes_R S$. But clearly ${}^\gamma f$ and $f$ agree on $\mu(\hA)$ because of (\ref{GI}). Thus $f\in \uG(S)^\Gamma$. We leave it to the reader to verify that these two processes are inverses of each other.
\end{proof}



By Theorem~\ref{MainLong} and Lemma~\ref{prp:Twi-Sch_lem} we see that to establish Theorem~\ref{MainShort}(1)  we must show that if $R_\K = \K[t^{\pm 1}],$ then {\it  for $\Gamma(\K)$ and $S_{(\K)}$ as in} Section \ref{Section:Affine}, we have
\begin{equation}\label{MainShortproof1}
\bG\big(S_{(\K)}\big)^{\Gamma(\K)} = \uG_{(\g,\sigma)}(R_\K) = \Hom_R(\hA,R_\K)
\end{equation}
where the last equality is given by the definition of $\uG_{(\g,\sigma)}$.
\medskip




Since $S_{(\K)}$ is a Galois extension of $R_\K$ with Galois group $\Gamma(\K)$, the same descent reasoning of Lemma~\ref{prp:Twi-Sch_lem} shows the existence of an $R_\K$-Hopf algebra $\widehat{A_{(K)}}$ such that
\begin{equation}
\bG\big(S_{(\K)}\big)^{\Gamma(\K)} = \Hom_{R_\K}( \widehat{A_{(\K)}},R_\K).
\end{equation}

Given the canonical isomorphisms $\Hom_R(\hA,R_\K) \simeq \Hom_{R_\K}(\hA \otimes_R R_\K,R_\K),$ to establish \eqref{MainShortproof1} amounts to proving the following.
\begin{lemma} \label{prp:base_change}
There exists a natural $R_\K$-algebra isomorphism $\widehat{A_{(\K)}} \simeq \hA \otimes_R R_\K$.
\end{lemma}
\begin{proof}
In what follows we will use the subscript $(\,)_\K$ to denote the base change from $R = \F[t^{\pm 1}]$ to  $\K[t^{\pm 1}]$. For example $A_\K = A \otimes_R \K[t^{\pm 1}]$.

If  the Galois extension $S_{(\K)}$ of $R_\K$ with group $\Gamma(\K)$ used in Section \ref{Section:Affine} is obtained from the $\F$-objects $S$ and $R$ and $\Gamma$ of the present section by the base change $\F\to\K$, that is $S_{(\K)}= S_\K$ and $\Gamma(\K) = \Gamma$, then clearly  $\widehat{A_{(\K)}} = \hA \otimes_R R_\K$. To see this just observe that the descent construction that produces $\widehat{A_{(\K)}}$ is given by an action of $\Gamma(\K)$ on $A_\K \otimes_{R_\K} S_{(\K)} \simeq (A \otimes_R S)\otimes_\F \K$, but this action takes place in $(A \otimes_R S) \otimes 1$ and commutes with the base changes to $\K$. The $R_\K$ descended object $\widehat{A_{(\K)}}$ is thus nothing but the descended $R$-object $\hA$ to which we apply the base change $\K/\F$. But
\[
\hA \otimes_\F \K \simeq \hA \otimes_R R \otimes_\F \K \simeq \hA \otimes_R R_\K.
\]

The only outstanding case is when $r=3$, $\xi \in \K$ and $\xi \notin \F.$ This case is different because $\Gamma = \rm{Gal}$$(S/R) \simeq$ $ S_3,$ while $\Gamma(\K) = {\rm Gal}(S_{(\K)}/R_\K) \simeq \Z/3\Z.$  By viewing $K$ as a field is an extension of $\F(\xi),$ a base change argument as above reduces the problem to the case $\K=\F(\xi)$, which we now address.

Since $\sigma$ generates our Galois group $\Gamma(\K)\simeq\Z/3\Z$, by definition $\widehat{A_{(\K)}}= \{ x\in A_\K\otimes_{R_\K}S_{(\K)} : x={}^\sigma x \}$. We thus need to show that the canonical map
\begin{equation}\label{end}
\hA\otimes_R \F(\xi)[t^{\pm1}] \,\,\longrightarrow\,\, \{ x\in A_\K \otimes_{R_\K} S_\K = A \otimes_R \F(\xi)[t^{\pm \frac{1}{3}}] : x = {}^\sigma x \}.
\end{equation}
is bijective. Recall that
\[
\hA \,\,=\,\, \{ x\in A\otimes_R \F(\xi)[t^{\pm1/3}] : x = {}^\gamma x \,\,\forall \gamma\in \Gamma\simeq S_3 \}.
\]
Since $\sigma$ is $\F(\xi)$-linear, \eqref{end} is clearly injective. To see that this map is surjective we reason as follows. Let $x\in A\otimes_{R_\K} S_\K \simeq \hA\otimes_{R_\K} S_\K$. Write
\[
x = x_1\otimes 1 + x_2\otimes \xi + x_3\otimes t^{\frac{1}{3}} + x_4\otimes \xi t^{\frac{1}{3}} + x_5\otimes t^{\frac{2}{3}} + x_6\otimes \xi t^{\frac{2}{3}}
\]
for some unique $x_1,\dots,x_6 \in \hA$. A simple calculation shows that ${}^\sigma x = x$ forces $x_3=x_4=x_5=x_6=0$. Our map is thus surjective.

This completes the proof of the Lemma and establishes Theorem~\ref{MainShort}(1).
\end{proof}

Let $R[\varepsilon]$ and $S[\varepsilon]$ be the rings of dual numbers of $R$ and $S$ respectively. We extend the action of $\Gamma$ to $S[\varepsilon]$ by fixing $\varepsilon$. Then $S[\varepsilon]$ is a Galois extension of $R[\varepsilon]$ with Galois group $\Gamma$.

The same reasoning used in Lemma~\ref{prp:Twi-Sch_lem} show that:
\begin{lemma} \label{prp:Twi-Sch_leme}
The restriction map
\[
\Hom_S(A\otimes_R S, S[\varepsilon]) \longrightarrow \Hom_R(\hA,S[\varepsilon]); \quad f\longmapsto f|_{\hA}
\]
induces an abstract group isomorphism $\uG(S[\varepsilon])^\Gamma \overset{\simeq}{\longrightarrow} \uG_{(\g,\sigma)}(R[\varepsilon])$.
\qed
\end{lemma}

Consider the exact sequence of abstract groups \cite{DG,Wat}
\begin{equation}\label{Lie}
0 \longrightarrow \fLie(\uG)(S) \longrightarrow \uG(S[\varepsilon]) \overset{\varepsilon\mapsto0}\longrightarrow \uG(S) \longrightarrow 1.\footnote{\, The group $\fLie(\uG)(S)$ is abelian and its customary to use additive notation for its composition law. The left exact $0$ is thus more natural than $1$. See \cite[II~\S4]{DG} for details.}
\end{equation}

By taking $\Gamma$-invariants in (\ref{Lie} and appealing to Lemmas~\ref{prp:Twi-Sch_lem} and~\ref{prp:Twi-Sch_leme}, we obtain
\begin{equation}\label{Lieinv}
0 \longrightarrow (\fLie(\uG)(S))^\Gamma \longrightarrow \uG_{(\g,\sigma)}(R[\varepsilon]) \longrightarrow \uG_{(\g,\sigma)}(R) \longrightarrow 1
\end{equation}
(the surjectivity follows from the existence of sections of the canonical morphisms $R[\varepsilon]\to R$ and $S[\varepsilon]\to S$).

Assume in what follows that we are in characteristic $0$, namely that $\F=\Q$. Then $\fLie(\uG)(S) = \g \otimes_\Q S$ so that (\ref{Lieinv}) yields.
\begin{lemma} \label{prp:Lie^Gamma}
There is a canonical $R$-Lie algebra isomorphism
$(\g\otimes_\Q S)^\Gamma \simeq \fLie(\uG_{(\g,\sigma)})(R)$.
\qed
\end{lemma}
\bigskip

Just as we write $R$ instead of $R_\Q$, for convenience we will henceforth denote $\g_\Q(X_N^{(r)})$ simply by $\g(X_N^{(r)})$.

Consider the $R$-Lie algebra $\Lie(\uG_{(\g,\sigma)}) = \fLie(\uG_{(\g,\sigma)})(R)$. By descent considerations, as an $R$-module $\Lie(\uG_{(\g,\sigma)})$ is projective of rank $\dim_\Q (\g)$ (in fact free of this rank since $R$ is a principal ideal domain). Consider the $R$-functor of Lie algebras
\begin{equation}\label{Lie}
{\Lie(\uG_{(\g,\sigma)})}_\fa: R' \longmapsto \Lie(\uG_{(\g,\sigma)}) \otimes_R R'.
\end{equation}
It is an affine scheme represented by the symmetric algebra of the $R$-dual of $\Lie(\uG_{(\g,\sigma)})$. Since $\uG_{(\g,\sigma)}$ is smooth (see Remark \ref{Reductive}) we have by \cite[II~\S4.8]{DG}  a canonical isomorphism
\begin{equation}\label{LieIso}
{\Lie(\uG_{(\g,\sigma)})}_\fa \simeq \fLie(\uG_{(\g,\sigma)}).
\end{equation}

The following result is the key to the parts of the Main Theorem that remain to be proved.
\begin{proposition}\label{RLie}
There exists a natural $R$-Lie algebra isomorphism $\g(X_N^{(r)}) \simeq \Lie(\uG_{(\g,\sigma)})$.
\end{proposition}
\begin{proof}
\cite[Cor.~4.10]{MorPiaShi21} explicitly describes, inside the $S$-Lie algebra $\g\otimes_\Q S$, the generators and relations of a $\Q$-Lie algebra $\mathcal{L}$ isomorphic to $ \g(X_N^{(r)})$. These generators and relations are exactly the same as those given in the definition of the affine algebra $ \g(X_N^{(r)})$. To reflect this fact, we say that $\mathcal{L}$ is naturally isomorphic to $ \g(X_N^{(r)})$. On the other hand, it is also shown that $\mathcal{L} = (\g\otimes_R S)^\Gamma$ as sets. This shows that $\mathcal{L}$ is stable under the scalar action of $R$, hence is an $R$-Lie algebra. By the last Lemma $\mathcal{L} \simeq \Lie(\uG_{(\g,\sigma)})$ as $R$-Lie algebras.
\end{proof}

\begin{proof}[Proof of Theorem~\ref{MainShort}(2)]
By \eqref{LieIso} for any ring extension $R'$ of $R$ there exists a canonical $R'$-Lie algebra isomorphism
\begin{equation}\label{canisoLie}
\Lie(\uG_{(\g,\sigma)}) \otimes_R R' \to \fLie(\uG_{(\g,\sigma)})(R').
\end{equation}

The infinite-dimensional $\Q$-Lie algebra $ \g(X_N^{(r)})$ is given by generators and relations corresponding to the affine Cartan matrix of type $X_N^{(r)}$. By \cite{MP} for any field extension $\K$ of $\Q$ there is a canonical $\K$-Lie algebra isomorphisms
\begin{equation}\label{a}
\g(X_N^{(r)}) \otimes_\Q \K \simeq \g_\K(X_N^{(r)}).
\end{equation}

Using the natural $R$-Lie algebra structure on $\g(X_N^{(r)})$ given in Proposition~\ref{RLie} and the canonical isomorphism
\begin{equation}\label{b}
\g(X_N^{(r)}) \otimes_\Q \K \simeq \g(X_N^{(r)}) \otimes_R R \otimes_\Q \K \simeq \g(X_N^{(r)}) \otimes_R R_\K
\end{equation}
gives $\g_\K(X_N^{(r)})$ a natural $R_\K$-Lie algebra structure. This construction is clearly functorial in $\K$.

By appealing to Lemma~\ref{prp:Lie^Gamma} and the  three Lie algebra isomorphisms of  (\ref{b}) we get
\begin{gather*}
\g_\K(X_N^{(r)}) \simeq \g(X_N^{(r)})\otimes_\Q\K \simeq \Lie(\uG_{(\g,\sigma)})\otimes_\Q\K \simeq \Lie(\uG_{(\g,\sigma)})\otimes_RR\otimes_\Q\K \\
\simeq \Lie(\uG_{(\g,\sigma)})\otimes_RR_\K \simeq \fLie(\uG_{(\g,\sigma)})(R_\K) = \fLie(\uG_{(\g,\sigma)})(\K[t^{\pm1}]).
\end{gather*}
This establishes the second part of the Theorem.

\subsection{}
It remains to show uniqueness. For convenience let us denote $\uG_{(\g,\sigma)}$ simply by $\fG$. Let $\fG'$ be a semisimple simply connected $R$-group. By \cite[Exp.~XXIV~Prop.~7.3.1]{SGA3} the affine $R$-schemes $\Isom(\fG, \fG')$ and $\Isom\big({\fLie({\fG}), \fLie(\fG')}\big)$ are canonically isomorphic. If in addition $\fG'$ satisfies Theorem~\ref{MainShort}(2), then the $R$-Lie algebras $\Lie(\fG) $ and $\Lie(\fG') $ are isomorphic. From \eqref{LieIso} it follows that $\Isom\big({\fLie(\fG), \fLie(\fG')}\big)(R) \neq \emptyset$. But then $\Isom(\fG, \fG')(R) \neq \emptyset$. Thus the $R$-group schemes $\fG$ and $\fG'$ are isomorphic.
\end{proof}

\begin{remark}\label{rationality} Let $\F$ be a prime field. It is natural to ask if the $\F[t^{\pm 1}]$-groups  $\uG_{(\g,\sigma)}$ that we have defined exists over $\Z[t^{\pm 1}].$ To be precise, the question is whether there exists a reductive group scheme $\fG$ over $\Z[t^{\pm 1}]$ with the property that $\fG \times_{\Z[t^{\pm 1}]} \F[t^{\pm 1}] \simeq \uG_{(\g,\sigma)}$.

If $r=1$ the answer is affirmative given that  $\uG_{(\g,\sigma)}$ already exists over $\Z$. Indeed $\uG_{(\g,\sigma)} \simeq \bG \times_\Z \F[t^{\pm 1}]$. The situation is entirely different for $r = 2,3$. For convenience let us denote $\Z [t^{\pm 1}]$ by $R_0$. We will show that a $\fG$ with the desired properties does not exist. For $s \in {\rm Spec}(R_0)$ let $\bar{\kappa}(s)$ be the algebraic closure of the residue field of $s$. The geometric fiber $\fG_{\bar{\kappa}(s)}$ is a (connected) reductive algebraic group over $\bar{\kappa}(s)$. The (isomorphism class) of a root datum for this group is called {\it the type of $\fG$ at $s$} (\cite[Exp.~XXII]{SGA3}). The type is locally constant, hence constant since $R_0$ is connected. By assumption $\fG$  becomes $\uG_{(\g,\sigma)}$ after a base change. Since type is invariant under an arbitrary base change, it follows that $\fG$ is of type $\bG$ (more precisely of type $(X_N, \Pi, \Pi^{\vee}, P)$, where $P$ is the weight lattice of $(\Delta, \Pi)$).

By \cite[Exp.~XXIV~Cor.~1.8]{SGA3} $\fG$ is a (twisted) form of $\bG_0$, where $\bG_0 = \bG \times _\Z R_0.$ Let $c(\fG) \in H^1_{\et}\big(R_0, \bAut(\uG)\big)$ be the corresponding element.\footnote{\, As customary we write $H^1_\et \big(R_0,\bAut(\uG)\big)$ instead of $H^1_\et \big(R_0,\bAut(\uG_{R_0})\big)$. Similarly for the other two terms of the sequence (\ref{exactH^1}).} 

By passing to cohomology on (\ref{exact}) we obtain the exact sequence of pointed sets
\begin{equation}\label{exactH^1}
H^1_\et (R_0, \bG_{\rm ad}) \longrightarrow H^1_\et \big(R_0,\bAut(\uG)\big) \overset{H^1(\rho)}{\longrightarrow} H^1_\et \big(R_0,\bOut(\g)\big) \longrightarrow 1
\end{equation}

By \cite[Exp.~XI~\S5]{SGA1}, $H^1_\et \big(R_0,\bOut(\g)\big)$ can be identified with the set of conjugacy classes of continuous homomorphisms of the algebraic fundamental group $\pi_1(R_0)$ of $R_0$ (at an arbitrary geometric base point) to the finite group $\Out(\g)$. Since $\pi_1(R_0)$ is trivial, we have $H^1_\et \big(R_0,\bOut(\g)\big) = 1$. As a consequence $H^1(\rho)\big(c(\fG)\big) = 1$, so that $\fG$ is an inner form of $\bG_{R_0}$.

By construction $\uG_{(\g,\sigma)}$ is a form of ${\bG_{R_0}}_{_\F}$, where ${R_0}_{\F} = R_0 \otimes_\Z \F = \F[t ^{\pm 1}].$ As above, it corresponds to an element $c(\uG_{(\g,\sigma)}) \in H^1_{\et}\big({R_0}_{_\F}, \bAut(\uG)\big)$. Since by assumption $\fG \times_{\Z[t^{\pm 1}]} \F[t^{\pm 1}] \simeq \uG_{(\g,\sigma)}$, by functoriality $c(\fG)$ is mapped to $c(\uG_{(\g,\sigma)})$ under the canonical map $H^1_{\et}\big(R_0, \bAut(\uG)\big) \to H^1_{\et}\big({R_0}_{_\F}, \bAut(\uG)\big)$. By functoriality again, $H^1(\rho)\big(c(\fG)\big) = 1$ implies that $H^1(\rho_\F)\big(c(\uG_{(\g,\sigma)})\big) = 1$. But this is a contradiction since by construction $\uG_{(\g,\sigma)}$ is not an inner form of $\bG_{R_\F}$.
\end{remark}

\noindent {\bf Acknowledgment} We would like to express our sincere gratitude to the referee for his/her thorough reading of our manuscript, and the many useful comments and suggestions.

\bigskip


\end{document}